\documentclass[12pt]{article}
\usepackage{array,amsmath,amssymb,amsthm,fullpage}
\usepackage{graphicx,subfigure}

\begin{document}

\newtheorem{thm}{Theorem}[section]
\newtheorem{lem}[thm]{Lemma}
\newtheorem{prop}[thm]{Proposition}
\newtheorem{cor}[thm]{Corollary}
\newtheorem{con}[thm]{Conjecture}
\newtheorem{claim}[thm]{Claim}
\newtheorem{obs}[thm]{Observation}
\newtheorem{ques}[thm]{Question}
\theoremstyle{definition}
\newtheorem{defn}[thm]{Definition}
\newtheorem{example}[thm]{Example}
\def\dfc{\mathrm{def}}
\def\cF{{\cal F}}
\def\cH{{\cal H}}
\def\cT{{\cal T}}
\def\cM{{\cal M}}
\def\cA{{\cal A}}
\def\cB{{\cal B}}
\def\cG{{\cal G}}
\def\ap{\alpha'}
\def\nul{\varnothing} %{\emptyset}
\def\st{\colon\,}   %"such that"
\def\MAP#1#2#3{#1\colon\,#2\to#3}
\def\VEC#1#2#3{#1_{#2},\ldots,#1_{#3}}
\def\VECOP#1#2#3#4{#1_{#2}#4\cdots #4 #1_{#3}}
\def\SE#1#2#3{\sum_{#1=#2}^{#3}}  \def\SGE#1#2{\sum_{#1\ge#2}}
\def\PE#1#2#3{\prod_{#1=#2}^{#3}} \def\PGE#1#2{\prod_{#1\ge#2}}
\def\UE#1#2#3{\bigcup_{#1=#2}^{#3}}
\def\UM#1#2{\bigcup_{#1\in #2}}
\def\FR#1#2{\frac{#1}{#2}}
\def\FL#1{\left\lfloor{#1}\right\rfloor} 
\def\CL#1{\left\lceil{#1}\right\rceil}  
\def\CH#1#2{\binom{#1}{#2}}  
\def\xr{\chi_r}
\def\ch{{\rm ch}}
\def\diam{{\rm diam}}
\def\ZZ{{\mathbb{Z}}}
\def\NN{{\mathbb{N}}}
\def\e{{\rm e}}
\def\esub{\subseteq}
\def\C#1{\left|#1\right|}
\def\arraystretch{.8}
\def\Imp{\Rightarrow}

\title{On $r$-dynamic Coloring of Grids}
\author{Ross Kang\thanks{Applied Stochastics, IMAPP, Radboud University
Nijmegen, Netherlands, ross.kang@gmail.com.  Research supported by a Veni
grant from the Netherlands Organization for Scientific Research (NWO).}\,,
Tobias M\"uller\thanks{Department of Mathematics, University of Utrecht,
Utrecht, Netherlands, t.muller@uu.nl.  Work of the first two authors done
while visiting Zhejiang Normal University.}\,, and
Douglas B. West\thanks{Departments of Mathematics, Zhejiang Normal University
and University of Illinois, Urbana, IL, USA, west@math.uiuc.edu.  Research
supported by Recruitment Program of Foreign Experts,
1000 Talent Plan, State Administration of Foreign Experts Affairs, China.
}
}

\maketitle

\begin{abstract}
An \textit{$r$-dynamic $k$-coloring} of a graph $G$ is a proper $k$-coloring of
$G$ such that every vertex in $V(G)$ has neighbors in at least $\min\{d(v),r\}$
different color classes.  The \textit{$r$-dynamic chromatic number} of a graph
$G$, written $\chi_r(G)$, is the least $k$ such that $G$ has such a coloring.
Proving a conjecture of Jahanbekam, Kim, O, and West, we show that the
$m$-by-$n$ grid has no $3$-dynamic $4$-coloring when $mn\equiv2\mod 4$.  This
completes the determination of the $r$-dynamic chromatic number of the
$m$-by-$n$ grid for all $r,m,n$.
\end{abstract}

\baselineskip=16pt           %line-spacing

\section {Introduction}
When proper graph colorings represent assignment of vertices to categories, in
some applications it is desirable for vertices to have neighbors in many
categories.  This increases the number of colors needed.

An {\it $r$-dynamic $k$-coloring} is a proper $k$-coloring $f$ of $G$ such that
$|f(N(v))| \ge \min\{r,d(v)\}$ for each vertex $v$ in $V(G)$, where $N(v)$ is
the neighborhood of $v$ and $f(U) = \{f(v)\st v \in U\}$ for a vertex subset
$U$.  The {\it r-dynamic chromatic number}, introduced by Montgomery~\cite{M}
and written as $\chi_r(G)$, is the least $k$ such that $G$ has an $r$-dynamic
$k$-coloring. 

Note that $\chi_1(G)=\chi(G)$, where $\chi(G)$ is the ordinary chromatic number
of $G$.  Montgomery called the 2-dynamic chromatic number simply the
{\it dynamic chromatic number}.  Many results were motivated by Montgomery's
conjecture that $\chi_2(G)\le\chi(G)+2$ when $G$ is regular, which remains open.
Bounds on $\chi_r$ and further references for work on $\chi_r$ and its
variations appear in~\cite{JKOW}.

In this note we complete the solution of a problem in~\cite{JKOW}.  For
$p\in\NN$, let $[p]=\{1,\ldots,p\}$.  The {\it $m$-by-$n$ grid} $G_{m,n}$ is
the graph with vertex set $[m]\times [n]$ such that $(i,j)$ and $(i',j')$ are
adjacent if and only if $\C{i-i'}+\C{j-j'}=1$.  (In more general language,
$G_{m,n}$ is the cartesian product of paths with $m$ and $n$ vertices.) The
problem of computing $\chi_r(G_{m,n})$ for all $r,m,n$ was proposed
in~\cite{JKOW}.

The following observations are immediate from the definition.

\begin{obs}\label{triv}
$\chi_{r+1}(G) \ge \chi_r(G)$.
\end{obs}

\begin{obs}\label{easy}
If $r\ge\Delta(G)$, then $\chi_r(G)=\chi_{\Delta(G)}(G)$.
\end{obs}

\begin{obs}\label{easy2}
$\chi_r(G)\ge\min\{\Delta(G),r\}+1$.
\end{obs}

To avoid trivialities, assume $m,n\ge2$.  Akbari, Ghanbari, and
Jahanbekam~\cite{ACJ2} proved $\chi_2(G_{m,n})=4$.  Jahanbekam, Kim, O, and
West~\cite{JKOW} then determined most of the other values.  Since
$\Delta(G_{m,n})\le 4$, by Observation~\ref{easy} we need only consider $r\le4$.

\begin{thm}[\cite{JKOW}]\label{oldvalue}
If $m$ and $n$ are at least 2, then 
$$
\chi_4(G_{m,n}) = \begin{cases} 4 & \text{if}~\min\{m,n\}=2\\
                                    5 & \text{otherwise} \end{cases}
~\text{and}~
\chi_3(G_{m,n}) = \begin{cases} 4 & \text{if}~\min\{m,n\}=2 \\
                           4 & \text{if}~m~\text{and}~n~\text{are both even.}\\ 
     5 & \text{if}~m,n~\text{not both even}~\text{and}~mn\not\equiv2\!\!\!\mod4
\end{cases}
$$
\end{thm}
The upper bounds in Theorem~\ref{oldvalue} are by explicit construction.
Setting $f(i,j)=i+2j\mod 5$ yields a $4$-dynamic $5$-coloring of $G_{m,n}$.
Constructions for $\chi_r(G_{m,n})\le 4$ are obtained by repeating (and
truncating when $m$ or $n$ is twice an odd number) the block below.
$$
\begin{array}{cccc} 0&1&2&3\\ 2&3&0&1\\ 1&0&3&2\\ 3&2&1&0 \end{array}
$$
The lower bounds in Theorem~\ref{oldvalue} are from Observation~\ref{easy2},
except when $m$ or $n$ is odd and the other is not twice an odd number.  In
that case, the lower bound follows from the discussion in Lemma~\ref{ring}.
A statement similar to Lemma~\ref{ring} is used in~\cite{JKOW} to prove
Theorem~\ref{oldvalue}.  We include a more explicit version of their
observations, because we will use them in proving our main result.  The
discussion also shows why the remaining case is harder, and it restricts the
configurations that need to be considered in that case.  For $mn\equiv2\mod4$
with $m,n\ge3$, the authors in~\cite{JKOW} proved that five colors suffice and
conjectured that five colors are needed.  The proof of this conjecture is the
content of this note.

\begin{thm}\label{main}
If $m,n\ge3$ and $mn\equiv2\mod4$, then $\chi_3(G_{m,n})=5$.
\end{thm}

\section{Preliminary Lemmas}\label{prodsec}

We henceforth assume $m,n\ge3$, with $m$ odd.  We represent a coloring of 
$G_{m,n}$ by a matrix $X$, with $x_{i,j}=f(i,j)$.  We use the four colors
$a,b,c,d$; their names may be permuted as needed, often invoked by saying
``by symmetry''.

In the statements of the lemmas, we choose $m$ to be a smallest odd integer
such that a $3$-dynamic $4$-coloring of $G_{m,n}$ exists for some $n$, and we
restrict the properties of such a coloring.  We say that a position {\it sees}
a color if it has a neighbor with that color; each position other than the four
corners must see three colors.  The {\it border vertices} are the vertices with
degree less than $4$.

The necessity of $n\equiv2\mod 4$ was obtained in \cite{JKOW}.  We give a more
explicit description of the coloring than they did, since we continue on to
obtain a contradiction.  As we have mentioned, the discussion in the proof of
Lemma~\ref{ring} is similar to \cite{JKOW}.

A portion of a row or column is {\it periodic} if vertices having the same
color are separated by a multiple of four positions.

\begin{lem}\label{ring}
The color sequences on the first two rows and first two columns are periodic.
Letting $a=x_{1,1}$, $b=x_{1,2}$, $c=x_{2,1}$, and $d=x_{2,2}$, the four colors
are distinct.  The cycle of colors is $(a,b,c,d)$ in the first row, $(c,d,a,b)$
in the second row, $(a,c,b,d)$ in the first column, and $(b,d,a,c)$ in the
second column.  Furthermore, $n\equiv2\mod4$ is necessary (given that $m$ is
odd), and columns $n-1$ and $n$ are copies of columns $1$ and $2$, respectively.
\end{lem}

\begin{proof}
Since border vertices have degree at most $3$, the colors
$x_{1,1},x_{1,2},x_{2,1},x_{2,2}$ are distinct.  Each noncorner border vertex
must see three colors.  Repeatedly using this observation determines the first
two rows and first two columns as claimed.  Once the argument for the first two
rows or first two columns reaches their ends, the same argument determines the
last two columns and last two rows.

We have restricted $m$ to be odd.  The diagram below, in the two cases
$m\equiv1\mod4$ and $m\equiv3\mod4$, incorporates all the cases for $(m,n)$.
In the bottom row the first two elements agree with the top row when
$m\equiv1\mod4$ and reverse those two elements when $m\equiv3\mod4$.  By
symmetry, the last two columns must exhibit the same behavior.
$$
\arraycolsep=1.8pt
\begin{array}{cccccccccccccc}
a&b&c&d&a&b&c&d&a&b&c&d&a&b\\
c&d&a&b&c&d&a&b&c&d&a&b&c&d\\
b&a&~&~&~&~&~&~&~&~&~&~&b&a\\
d&c&a&b&d&c&a&b&d&c&a&b&d&c\\
a&b&d&c&a&b&d&c&a&b&d&c&a&b\\
~&~&~&1&1&~&4&4&~&3&3&~&2&2
\end{array}
\qquad
\begin{array}{cccccccccccccc}
a&b&c&d&a&b&c&d&a&b&c&d&a&b\\
c&d&a&b&c&d&a&b&c&d&a&b&c&d\\
b&a&~&~&~&~&~&~&~&~&~&~&b&a\\
d&c&~&~&~&~&~&~&~&~&~&~&d&c\\
a&b&~&~&~&~&~&~&~&~&~&~&a&b\\
c&d&b&a&c&d&b&a&c&d&b&a&c&d\\
b&a&c&d&b&a&c&d&b&a&c&d&b&a\\
~&~&~&1&1&~&4&4&~&3&3&~&2&2
\end{array}
$$

The numbers below the grid designate where the rows end when $n$ is congruent
to $1$, $4$, $3$, or $2$, respectively.  In the first three cases, the
relationship between the top row and bottom row is not what we have observed
is necessary for the last two columns, so the rows cannot end there.  Hence in
those congruence classes for $n$ no $3$-dynamic $4$-coloring exists.
\end{proof}

Thus we henceforth also assume $n\equiv2\mod4$, with $n\ge6$.

\begin{lem}\label{cycle}
In the upper left corner, $x_{3,2}=x_{2,3}$, and similarly in the other corners.
Furthermore, $m\ge7$.
\end{lem}
\begin{proof}
The cycling of colors as observed in Lemma~\ref{ring} implies the first claim.
We then observe that when $m\in\{3,5\}$, position $(3,3)$ cannot see three
colors.
\end{proof}

Additional lemmas will restrict the coloring of the interior.

\begin{lem}\label{2by3}
Every $3$-by-$2$ or $2$-by-$3$ subgrid has vertices of all four colors.
\end{lem}

\begin{proof}
By symmetry, it suffices to consider a $2$-by-$3$ subgrid.  Suppose that some
$2$-by-$3$ subgrid contains at most three colors in a $3$-dynamic $4$-coloring
of $G_{m,n}$.  Since the $6$-vertex subgraph is bipartite, a color appearing
three times would give a vertex three neighbors with the same color, which is
forbidden since each vertex has degree at most $4$.  Hence each of the three
colors appears twice.  By symmetry, we may assume that it appears as follows.
$$
\begin{array}{ccc}
a&b&c\\
b&c&a\end{array}
$$

Since border vertices have degree at most $3$, such a subgrid cannot include a
border vertex.  Now the fourth color must appear above and below the middle,
but it also must appear adjacent to each corner of the rectangle, which puts it
on adjacent vertices.
\end{proof}

We next prove another completely local implication.

\begin{lem}\label{zigzag}
If $x_{i-1,j}=x_{i,j-1}$ and $x_{i,j}\ne x_{i-1,j+1}$, then
$x_{i+1,j}=x_{i-1,j+1}$ and $x_{i+1,j+1}=x_{i-1,j}$ (given $i<m$ and $j<n$).
The same implication holds with rows and columns exchanged.
\end{lem}
\begin{proof}
By symmetry, we may assume $a=x_{i-1,j}=x_{i,j-1}$.  Since $(i,j)$ and
$(i-1,j+1)$ both see color $a$ and are assumed to have different colors, by
symmetry we may let $b=x_{i-1,j+1}$ and $d=x_{i,j}$, as shown below.  Now
$x_{i,j+1}\notin\{a,b,d\}$, so $x_{i,j+1}=c$, which in turn yields
$x_{i+1,j}=b$.
$$
\begin{array}{ccc} ~&a&b\\ a&d&~\\ ~&~&~ \end{array}
\quad\Imp\quad
\begin{array}{ccc} ~&a&b\\ a&d&c\\ ~&b&~ \end{array}
\quad\Imp\quad
\begin{array}{ccc} ~&a&b\\ a&d&c\\ ~&b&a \end{array}
$$
Now $x_{i+1,j+1}\in\{a,d\}$.  If $x_{i+1,j+1}=d$, then having $(i,j+1)$
and $(i+1,j)$ both see $a$ requires $x_{i+2,j}$ and $x_{i,j+2}$ to exist
and equal $a$, but then $(i+1,j+1)$ cannot see $a$.
\end{proof}

\begin{lem}\label{prop}
Suppose that row $r$ is periodic from columns $s$ to $t$.  If
$x_{r+1,s}=x_{r,s+1}$ and $x_{r+1,s+1}=x_{r,s+2}$, then $x_{r+1,j}=x_{r,j+1}$
for $s\le j<t$ (similarly for columns).  This situation cannot occur when $r=2$
or $s=2$, or for any $r$ when $x_{r+1,t-1}\ne x_{r,t}$ is known.
\end{lem}
\begin{proof}
By symmetry, we may assume that this periodic portion in row $r$ begins
$d,a,b,c$.  There is nothing to prove unless $t\ge s+3$.  In that case (shown
below), $x_{r,s}\ne x_{r,s+3}$.  Lemma~\ref{2by3} then implies
$x_{r+1,s+2}=x_{r,s+3}$.  This establishes the same conditions for the next
pair of columns.  Continuing to apply Lemma~\ref{2by3} copies row $r$ through
column $t$ into row $r+1$, shifted by one column.
$$
\begin{array}{c|ccccccc}
&s&s\!+\!1&~&~&~&t&\\
\hline
r&d&a&~b~&~c~&~d~&~a~\\
r+1&a&b&~&&~&~&~
\end{array}
\quad\Imp\quad
\begin{array}{c|ccccccc}
&s&s\!+\!1&~&~&~&t&\\
\hline
r&d&a&~b~&~c~&~d~&~a~\\
r+1&a&b&c&d&a&~&~
\end{array}
$$

Because $(3,n)$ has only three neighbors, $x_{3,n-1}\ne x_{2,n}$.  Also row $2$
is periodic up to $t=n$, as shown in Lemma~\ref{ring}.  Hence the given fact
that $x_{r,t}$ does not copy into position $(r+1,t-1)$ (in row $2$ or later)
prevents row $r+1$ from having two consecutive positions copied from the
periodic portion of row $r$ (shifted back by one position).  The statement and
proof are symmetric for columns.
\end{proof}

\begin{defn}\label{cohere}
A $4$-coloring of $G_{m,n}$ is {\it coherent} if the colors on the $4$-by-$4$
grid in the upper left have the form of the matrix below (under any permutation
of the four colors).
$$
\begin{array}{cccc}
a&b&c&d\\
c&d&a&b\\
b&a&d&c\\
d&c&b&a
\end{array}
$$
\end{defn}

\begin{lem}\label{corner}
Any $3$-dynamic $4$-coloring of $G_{m,n}$ is coherent.
\end{lem}
\begin{proof}
So far we have the pattern of the border positions and their neighbors, as
indicated in Lemma~\ref{ring}.  Applying Lemma~\ref{prop} to row $3$ yields
$x_{3,3}\ne b$, and applying it to column $3$ yields $x_{3,3}\ne c$, so
$x_{3,3}=d$.  Now Lemma~\ref{zigzag} yields $x_{4,3}=b$ and $x_{4,4}=a$, and
then $x_{3,4}=c$.
\end{proof}
$$
\begin{array}{|cccccc}
\hline
a&b&c&d&a&b\\
c&d&a&b&c&d\\
b&a&d&c&~&~\\
d&c&b&a&~&~\\
a&b&~&~&~&~\\
c&d&~&~&~&~
\end{array}
$$

\section{Proof of Theorem~\ref{main}}

We have seen that the colors cycle in the first two rows and in the first
two columns.  Since the columns have odd length, the third and fourth columns
cannot cycle all the way to the bottom.  When the first row starts $(a,b,c,d)$,
such cycling would leave the bottom row starting $(a,b,c,d)$ if $m\equiv1\mod4$
or $(b,a,d,c)$ if $m\equiv3\mod4$.  This would contradict the patterns in
Lemma~\ref{ring}, which showed that these colors are $(a,b,d,c)$ or
$(b,a,c,d)$, respectively.  Hence cycling must stop somewhere in column $3$
or $4$.

Cycling also cannot continue all the way across rows $3$ and $4$.  If it did,
then we could delete rows $1$ and $2$ to obtain a $3$-dynamic $4$-coloring of
$G_{m-2,n}$, contradicting the minimality of $m$ and thereby completing the
proof.

We introduce definitions to facilitate study of where cycling of colors breaks
down.

\begin{defn}\label{correct}
Given $(i,j)\in[m]\times[n]$, define $(p,q)\in[4]\times[4]$ by
$i\equiv p\mod4$ and $j\equiv q\mod 4$.
A position $(i,j)$ is {\it correct} if $x_{i,j}=x_{p,q}$.
For $i>4$, positions $(i,j)$ and $(i+1,j)$ are {\it flipped} if
they are incorrect but $x_{i,j}=x_{p+1,q}$ and $x_{i+1,j}=x_{p,q}$.
Similarly, for $j>4$, positions $(i,j)$ and $(i,j+1)$ are {\it flipped} if
$x_{i,j}=x_{p,q+1}$ and $x_{i,j+1}=x_{p,q}$.
\end{defn}

We observed at the beginning of this section that neither rows $3$ and $4$ nor
columns $3$ and $4$ are completely correct.  Hence there is a first column
having an incorrect position in row $3$ or $4$, and there is a first row having
an incorrect position in column $3$ or $4$.  We next show that these first
incorrect positions are in row $3$ and column $3$, respectively.
Note that correctness requires $x_{3,j}=x_{2,j+1}$ when $j$ is even,
but these positions have different colors when $j$ is odd.  We also show that
the first violation of correctness leads to flipped positions.

\begin{lem}\label{extend}
Given $4\le2s\le n-2$, suppose that row $2r$ is periodic through column $t$
such that $t>2s$ and $x_{2r+1,t-1}\ne x_{2r,t}$.  If rows $2r+1$ and $2r+2$ are
correct through column $2s$, then either they are both correct through column
$2s+2$, or positions $(2r+2,2s+1)$ and $(2r+2,2s+2)$ are correct while positions
$(2r+1,2s+1)$ and $(2r+1,2s+2)$ are flipped.  The same statement holds with the
roles of rows and columns switched.
\end{lem}
\begin{proof}
Let the cycle of colors in row $2r$ have $(a,b,c,d)$ ending at column $2s$.
$$
\begin{array}{c|cccccccc}
~&~&~&~&~&~&2s&~&~\\
\hline
~&a&b&c&d&a&b&~&~\\
2r&c&d&a&b&c&d&a&b\\
~&b&a&d&c&b&a&~&~\\
~&d&c&b&a&d&c&~&~\\
\end{array}
$$
Since $(2r+1,2s)$ is correct and $t>2s$, we have $x_{2r+1,2s}=x_{2r,2s+1}=a$
and $x_{2r,2s+2}=b$.  By Lemma~\ref{prop}, $x_{2r+1,2s+1}\ne b$.  Whether
$(2r+1,2s+1)$ is correct or not, Lemma~\ref{zigzag} now yields
$x_{2r+2,2s+1}=b$ and $x_{2r+2,2s+2}=a$, which
makes $(2r+2,2s+1)$ and $(2r+2,2s+2)$ correct.

If $x_{2r+1,2s+1}=d$, then $(2r+1,2s+2)$ sees $\{a,b,d\}$.  Hence
$x_{2r+1,2s+2}=c$, making $(2r+1,2s+1)$ and $(2r+1,2s+2)$ also correct.

Otherwise, $x_{2r+1,2s+1}=c$.  Now $(2r+1,2s+1)$ sees $\{a,b,c\}$, so
$x_{2r+1,2s+1}=d$ and positions $(2r+1,2s+1)$ and $(2r+1,2s+2)$ are flipped.
\end{proof}

Applying Lemma~\ref{extend} with $r=1$ yields a value $s$ (with $s\ge2$) such
that row $3$ is correct through column $2s$ with $(3,2s+1)$ and $(3,2s+2)$
flipped, and row $4$ is correct through column $2s+2$.  Applying the column
version with $s=1$ also yields a value $r$ (with $r\ge 2$) such that column $3$
is correct through row $2r$ with $(2r+1,3)$ and $(2r+2,3)$ flipped, and column
$4$ is correct through row $2r+2$.

\begin{defn}\label{dfpartial}
With respect to a coherent block in the upper left, an {\it $(r,s)$-partial
coloring} of a grid is a vertex $4$-coloring that flips positions $(3,2s+1)$
and $(3,2s+2)$ and flips positions $(2r+1,3)$ and $(2r+2,3)$, but is correct
on all of the following: the first three rows through column $2s$, row $4$
through column $2s+2$, the first three columns through row $2r$, and column $4$
through row $2r+2$.
\end{defn}

We remarked before Definition~\ref{dfpartial} that every $3$-dynamic
$4$-coloring of $G_{m,n}$ is an $(r,s)$-partial coloring for some $(r,s)$.
Hence our next lemma completes
the proof of the theorem.

\begin{lem}\label{partial}
For $r,s\ge2$, an $(r,s)$-partial coloring of a grid cannot be completed to a
$3$-dynamic $4$-coloring.
\end{lem}
\begin{proof}
We use induction on $r+s$.  For $r=s=2$, we have the coloring shown below.
It cannot be completed, because $(5,4)$ requires $x_{5,5}=c$, but $(4,5)$
requires $x_{5,5}=b$.
$$
\begin{array}{cccccc}
a&b&c&d&~&~\\
c&d&a&b&~&~\\
b&a&d&c&a&b\\
d&c&b&a&d&c\\
~&~&a&d&~&~\\
~&~&c&b&~&~
\end{array}
$$

Now suppose $r+s>4$.  Since $(3,2s+1)$ and $(3,2s+2)$ are flipped, 
$x_{4,2s}=x_{3,2s+1}$.  Since $(4,2s+1)$ and $(4,2s+2)$ are correct,
$x_{4,2s+1}\ne x_{3,2s+2}$.  Hence Lemma~\ref{zigzag} yields
$x_{5,2s+1}=x_{3,2s+2}$ and $x_{5,2s+2}=x_{3,2s+1}$, which means that
$(5,2s+1)$ and $(5,2s+2)$ are flipped.  Similarly, $(2r+1,5)$ and $(2r+2,5)$
are flipped.
$$
\begin{array}{c|cccccccccccc|}
~&1&2&3&~&~&~&~&~&~&2s&~&~\\
\hline
1&~&~&~&~&~&~&~&~&~&~&~&~\\
2&~&~&~&~&~&~&~&~&~&~&~&~\\
3&~&~&a&b&c&d&a&b&c&d&b&a\\
4&~&~&c&d&a&b&c&d&a&b&c&d\\
5&~&~&b&a&~&~&~&~&~&~&a&b\\
~&~&~&d&c&~&~&~&~&~&~&~&~\\
~&~&~&a&b&~&~&~&~&~&~&~&~\\
2r&~&~&c&d&~&~&~&~&~&~&~&~\\
~&~&~&d&a&b&~&~&~&~&~&~&~\\
~&~&~&b&c&d&~&~&~&~&~&~&~\\
\hline
\end{array}
$$

In the matrix above, we have applied a permutation of the colors to illustrate
the smaller instance of the problem to which we will apply the induction 
hypothesis.  Assume these labels.  If $x_{5,5}=b$, then Lemma~\ref{prop}
applies in row $5$ to contradict $x_{5,2s+1}\ne x_{4,2s+2}$.  If $x_{5,5}=c$,
then Lemma~\ref{prop} applies in column $5$ to contradict
$x_{2r+1,5}\ne x_{2r+2,4}$.  Hence $x_{5,5}=d$.  Now Lemma~\ref{zigzag} yields
$x_{6,5}=b$, $x_{6,6}=a$, and $x_{5,6}=c$.  Furthermore, the positions in
$\{5,6\}\times\{5,6\}$ are correct.  (This argument also implies that $r>2$ if
and only if $s>2$.)

We now build a smaller instance of the problem.  Since positions $(5,2s+1)$ and
$(5,2s+2)$ are flipped and $x_{5,2s+1}\ne x_{4,2s+2}$, Lemma~\ref{extend}
implies that for some $s'$ with $s'\le s$, rows $5$ and $6$ are correct through
column $2s'$, with positions $(6,2s'+1)$ and $(6,2s'+2)$ correct while
$(5,2s'+1)$ and $(5,2s'+2)$ are flipped.  The same discussion applies to
columns $5$ and $6$, yielding $r'$ with $r'\le r$.  Note that $r',s'\ge2$.

Let $M$ be the matrix obtained by deleting the first two rows and first two
columns of the given coloring.  Since the positions in $\{5,6\}\times\{5,6\}$
are correct, $M$ is coherent.  Hence $M$ is an $(r'-2,s'-2)$-partial coloring
of a grid.  By the induction hypothesis, it cannot be extended to a $3$-dynamic
$4$-coloring.  Since $M$ extends to row $m$ and column $n$ of the original
coloring, we conclude that the original coloring also does not extend to a
$3$-dynamic $4$-coloring.
\end{proof}

\end{document}